\documentclass[a4paper,12pt]{article}
\usepackage{amsmath,amssymb,amsthm}
\usepackage{graphicx}
\usepackage{color}
 \newtheorem{theorem}{Theorem}
 
 \newtheorem{claim}[theorem]{Claim}

 {\theoremstyle{definition}
 \newtheorem{definition}[theorem]{Definition}
 \newtheorem{remark}[theorem]{Remark}

 }

\def\N{\ensuremath{\mathbb N}} 
\def\T{\ensuremath{\mathbb T}} 
\def\R{\ensuremath{\mathbb R}} 
\def\C{\ensuremath{\mathbb C}} 
\def\Z{\ensuremath{\mathbb Z}} 

\newcommand{\GGG}{\Gamma}

\newcommand{\aaa}{\alpha}

\newcommand{\bbb}{\beta}

\newcommand{\ds}{\displaystyle}

\newcommand{\kkk}{\kappa}

\newcommand{\oo}{\infty}

\newcommand{\sse}{\subset}

\renewcommand{\ggg}{\gamma}

\title{Convergence of ergodic averages for many group rotations}
\author{
Zolt\'an Buczolich\thanks{This
author  was supported by the Hungarian
National Foundation for Scientific Research K075242.},
Department of Analysis, E\"otv\"os Lor\'and\\
University, P\'azm\'any P\'eter S\'et\'any 1/c, 1117 Budapest, Hungary\\
email: buczo@cs.elte.hu\\
{\tt www.cs.elte.hu/\hbox{$\sim$}buczo}\\
 \\ and\\
 \\
Gabriella Keszthelyi\thanks{This
author  was supported by the Hungarian
National Foundation for Scientific Research K104178.
 },
 Department of Analysis, E\"otv\"os Lor\'and\\
University, P\'azm\'any P\'eter S\'et\'any 1/c, 1117 Budapest, Hungary\\
email: keszthelyig@gmail.com\\
}

\begin{document}
\maketitle

\begin{abstract}
Suppose that $G$ is a compact Abelian topological group,
$m$ is the Haar measure on $G$ and 
 $f:G \rightarrow \R$ is a measurable function.
 Given $(n_k)$, a strictly monotone increasing sequence of integers  we consider the nonconventional ergodic/Birkhoff averages\\
$\ds M_N^{\alpha}f(x)=\frac{1}{N+1} \sum_{k=0}^N f(x+n_k{\alpha}).$\\
The $f$-rotation set is\\
$\ds \Gamma_f=\{\alpha \in G: M_N^{\alpha} f(x) \text{ converges for $m$ a.e. $ x$ as $N\to\oo$.}\}$

We prove that
if  $G$ is a compact locally connected Abelian group and $f: G \rightarrow \R$ is a measurable function then from $m(\Gamma_{f})>0$ it follows that $f \in L^1(G).$

A similar result is established for ordinary Birkhoff averages if $G=Z_{p}$, the group
of $p$-adic integers. 

However, if the dual group, $\widehat{G}$ contains ``infinitely
many multiple torsion" then such results do not hold if one considers non-conventional 
Birkhoff averages along ergodic sequences.

What really matters in our results is the boundedness of the tail, $f(x+n_{k}\aaa)/k$, $k=1,...$ for a.e. $x$ for many $\aaa$, hence some of our theorems are stated by using 
instead of $\GGG_{f}$
slightly larger sets, denoted by $\GGG_{f,b}$.
\end{abstract}


\section{Introduction}\label{introduction}

The starting point of this paper is a result of the first listed author in \cite{ETDS} which states that if $f$ is a (Lebesgue) measurable function on the unit circle $\T$
and $\GGG_{f}$ denotes the set of those $\aaa$'s for which the Birkhoff averages
$$ { M_ { n } ^ { \alpha } f }(x)={1\over n+1}
\sum_{k=0}^{n}f(x+k\alpha )$$
converge for almost every $x$ then from $m(\GGG_{f})>0$ it follows that $f\in L^{1}(\T)$. Hence $M_{n}^{\aaa}f$ converges for all $\aaa\in\T$.

In this paper we consider generalizations of this result to compact Abelian groups 
equipped with their Haar measure $m$. Theorem \ref{thconn} implies that an analogous result is true even for non-convential ergodic averages considered on a compact,
locally connected Abelian group $G$.

On the other hand, if there is ``sufficiently many multiple torsion" in the dual group $\widehat{G}$
then Theorem \ref{thtorsion2} implies that there are non-$L^{1}$ measurable functions $f$
for which $m(\GGG_{f})=1$ (in fact, $\GGG_{f}=G$) if one considers non-conventional
Birkhoff averages along ergodic sequences. Having   lots of torsion in $\widehat{G}$
means that $G$ is highly disconnected.
In our opinion the most surprizing result of this paper is Theorem \ref{thpadic} which states that if $G=Z_{p}$, the group of $p$-adic integers and one considers the ordinary
ergodic averages of a measurable function $f$ then from $m(\GGG_{f})>0$
it follows that $f\in L^{1}(G)$. The group $Z_{p}$ is zero-dimensional and all elements
of its dual group, $Z(p^{\oo})$, are of finite order. If one considers a group $G$ which
is a countable product of $Z_{p}$'s then there is enough ``multiple torsion"
(see Definition \ref{deftor}) in $\GGG_{f}$ and Theorem \ref{thtorsion2} implies 
that the result of Theorem \ref{thpadic} does not hold in these groups.
If ${ M_ { n } ^ { \alpha } f }(x)$ converges then the tail $\ds \frac{f(x+n\aaa)}{n}\to 0$.
In our proofs the sets $\GGG_{f,0}$ (and $\GGG_{f,b}$), the sets of those $\aaa$'s where $\ds \frac{f(x+n\aaa)}{n}\to 0$, (or $\ds \frac{|f(x+n\aaa)|}{n}$ is bounded) for a.e. $x$ play an important role. Since $\GGG_{f}\sse \GGG_{f,0}\sse \GGG_{f,b}$ from $m(\GGG_{f})>0$ it follows that the other sets are also of
positive measure and hence in the statements of Theorems \ref{thconn} and \ref{thpadic}
these sets are used.  Again the tail of the ergodic averages plays an important role, like in
\cite{ABM}, where we showed that for $L^{1}$ functions and ordinary ergodic averages
the return time property for the tail may might fail and hence Bourgain's return time property
\cite{[bou]} does not hold in these situations.

The proof of Theorem \ref{thconn} is a rather straightforward generalization of
Theorem 1 in \cite{ETDS}. We provide its details, since they are also used with some
non-trivial modifications in the proof of Theorem \ref{thpadic}.

Next we say a few words about the background history and related questions to this paper.
Answering a question raised by the first listed author of this paper P. Major in \cite{Ma} constructed two
ergodic transformations 
 $S,T:X\to X$ 
on a probability space $(X,\mu)$
and a measurable function $ f:X\to  { \ensuremath { \mathbb R } }$ 
such that for $\mu$ a.e. $x$
$$\lim_{n\to { \infty }}{
\frac{1}{n+1}}\sum_{k=0}^{n}f(S^{k}x)= 0,\ \text{ and }
\lim_{n\to { \infty }}{\frac 1 {n+1}}
 \sum_{k=0}^{n}f(T^{k}x)= a\not=0.$$

M. Laczkovich  raised the question
 whether
$S$ and  $T$ can be
 irrational rotations of  $ { \ensuremath { \mathbb T } }$.
In Major's  example $S$ and $T$ are conjugate.
Therefore, his method did not provide an answer to Laczkovich's
question.

The results of Z. Buczolich in \cite{ego} imply that for any two
independent irrationals $\aaa$ and $\bbb$ one can find a measurable
$f:\T\to\R$ such that ${ M_ { n } ^ { \alpha } f }(x)\to c_{1}$
and ${ M_ { n } ^ { \beta } f }(x)\to c_{2}$ for a.e. $x$ with $c_{1}\not=c_{2}$.
In this case by Birkhoff's ergodic theorem $f\not\in L^{1}(\T)$.
It is shown in \cite{ETDS} that for any sequence $(\aaa_{j})$ of independent
irrationals  one can find a measurable $f:\T\to\R$ such that $f\not\in L^{1}(\T)$,
but $\aaa_{j}\in\GGG_{f}$ for all $j=1,...$. By Theorem 1 of \cite{ETDS} from
$f\not\in L^{1}(\T)$ it follows that $m(\GGG_{f})=0$. It was a natural question
to see how large $\GGG_{f}$ could be for an $f\not\in L^{1}(\T).$
In \cite{Sv} R. Svetic showed that $\GGG_{f}$ can be $c$-dense for an $f\not\in L^{1}(\T).$

The question about the possible largest Hausdorff dimension of $\GGG_{f}$ for an
$f\not\in L^{1}(\T)$ remained open for a while until in \cite{rot} it was shown
that there are $f\not\in L^{1}(\T)$ such that $\dim_{H}(\GGG_{f})=1$
(of course with $m(\GGG_{f})=0$.)

For us motivation to consider non-conventional ergodic averages in this paper
came from the project in \cite{SQA} concerning almost everywhere convergence questions
of Birkhoff averages along the squares.

It is also worth mentioning that ergodic averages of non-$L^{1}$ functions and
rotations on $\T$ were also considered in \cite{[SU1]} and \cite{[SU2]}.

\section{Preliminaries}
We suppose that $G$ is a compact Abelian topological group, the group operation will be addition. 
 The dual group of the compact Abelian topological group $G$ is denoted by $\widehat{G}$. By Pontryagin duality $\widehat{G}$ is a discrete Abelian group. For $\gamma \in \widehat{G}$ the corresponding Fourier coefficient is $$\widehat{f}(\gamma)=\int_G g(x) \gamma(-x) dm(x),$$ where $m$ denotes the Haar mesure on $G$. By the Parseval formula
$$\int_G f(x)\bar{g}(x)dm(x)=\sum_{\gamma \in \widehat{G}} \widehat{f}(\gamma)\overline{\widehat{g}(\gamma)}\ \text{ for $f,g \in L^2(G)$.}$$ 
By   \cite[24.25]{HR}  or \cite[2.5.6 Theorem]{Rudin}   if $G$ is a compact Abelian group then $G$ is connected if and only if $\widehat{G}$ is torsion-free. 

Suppose that $p_1, p_2,...$ is a sequence of prime numbers. Recall that the direct product $G=(Z/p_1)\times (Z/p_2)\times \dots$ is compact and its dual group $\widehat{G}=(Z/p_1)\bigoplus(Z/p_2) \bigoplus \dots$ is the direct sum with the discrete topology see 
\cite[2.2 p.36]{Rudin}
 or \cite{HR}.

 We denote by $Z_p$ the group of $p$-adic integers and its dual group, the Pr\"ufer $p$-group with the discrete topology will be denoted by $Z(p^{\infty})$. 
 
 For other properties of topological groups we refer to standard textbooks like \cite{Fuchs},
 \cite{HR} or \cite{Rudin}.
 
 Suppose that $f:G \rightarrow \R$ is a measurable function. We suppose that the group rotation $T_{\alpha}=x+\alpha, \; \  \alpha \in G$ is fixed. 
 
 Given a strictly monotone increasing sequence of integers $(n_k)$ we consider the nonconventional ergodic averages
$$M_N^{\alpha}f(x)=\frac{1}{N+1} \sum_{k=0}^N f(x+n_k{\alpha}).$$
Of course, if $n_k=k$ we have the usual Birkhoff averages. 

The $f$-rotation set is
$$\Gamma_f=\{\alpha \in G: M_N^{\alpha} f(x) \text{ converges for $m$ a.e. $ x$ as $N\to\oo$}\}.$$
As we mentioned in the introduction it
 was proved in \cite{ETDS} that if $G=\T , \: m=\lambda$, the Lebesgue measure on $\T $, and $n_k=k$ then for any measurable $f:\T  \rightarrow \R$ from $m(\Gamma_f)>0$ it follows that $f\in L^1(\T )$.

Scrutinizing the proof of this result one can see that the set
$$\Gamma_{f,0}=\Big \{\alpha \in G: \frac{f(x+n_k \alpha)}{k} \rightarrow 0 \text{ for $m$ a.e. $x$}\Big \}$$
played an important role.
It is obvious  that $\Gamma_f \subset \Gamma_{f,0}$. 

In \cite{ETDS} it was shown that from $m(\Gamma_{f,0})>0$ it follows that $f \in L^1(\T )$, when the sequence $n_k=k$ is considered. In this paper we will also use the slightly larger set
\begin{equation}\label{d3a}
\Gamma_{f,b}=\Big \{\alpha \in G: \limsup_{k \rightarrow \infty}\frac{|f(x+n_k \alpha)|}{k} < \infty  \text{ for $m$ a.e. $x$}\Big \}.
\end{equation}

\section{Main results}

First we generalize Theorem 1 of \cite{ETDS} for compact, locally connected Abelian groups.

\begin{theorem}\label{thconn} If $(n_k)$ is a strictly monotone increasing sequence of integers and $G$ is a compact, locally connected Abelian group and $f: G \rightarrow \R$ is a measurable function then from $m(\Gamma_{f,b})>0$ it follows that $f \in L^1(G).$
\end{theorem}

\begin{remark} Since $\Gamma_{f,b} \supset \Gamma_{f,0} \supset \Gamma_f$ Theorem \ref{thconn} implies that if one considers the non-conventional ergodic averages $M_N^{\alpha} f$ on a locally compact Abelian group for group rotations and $m(\Gamma_f)>0$ then $f \in L^1(G)$.
\end{remark}
\begin{proof} Set $n_0=0.$ First we suppose that $G$ is connected. Given an integer $K$ put
\begin{equation}\label{D4a}
G_{\alpha,K}=\{x:|f(x+n_k \alpha)|<K \cdot k \text{ for every } k>K 
\end{equation}
$$\text{ and } |f(x+n_k \alpha)|<K^2 \text{ for } k=0, \dots, K\}.$$
If $\alpha \in \Gamma_{f,b}$ then $m(G_{\alpha,K} )\rightarrow 1$ as $K \rightarrow \infty$. 

Choose and fix $K$ and $\varepsilon>0$ such that the set
\begin{equation}\label{D4b}
B=\{\alpha:m(G_{\alpha,K})> \varepsilon\}
\end{equation}
is of positive $m$-measure.
From the measurability of $f$ it follows that $B$ and the sets $g_{\alpha,K}$ are also measurable. 

Set
\begin{equation}\label{D4c}
L_k(f)=\{x \in G : |f(x)|>k\}.
\end{equation}
From $k>K$ and $x \in G_{\alpha,K}+n_k \alpha$ it follows that 
$$|f(x)|=|f(x-n_k\alpha+n_k\alpha)|<k\cdot K.$$ Set $H_{\alpha}=G\backslash G_{\alpha,K}$, (keep in mind that $K$ is fixed).
From $k>K$ and $x \in L_{k \cdot K} (f)$ it follows that $x \notin G_{\alpha,K}+n_k \alpha$, that is, $x \in H_{\alpha}+n_k \alpha.$ 

For $\alpha \in B$ we set $a(\alpha)=m(H_{\alpha})<1-\varepsilon$, by \eqref{D4b}. This implies $1/(1-a(\alpha))<1/\varepsilon$. 

For $\alpha \in B$ put
\begin{equation}\label{D5b}
h(x,\alpha)=\left\{
              \begin{array}{ll}
                1 \text{ if } x \in H_{\alpha}, & \hbox{} \\
                -\left(\frac{a(\alpha)}{1-a(\alpha)}\right) \text{ if } x \notin H_{\alpha}.& \hbox{}
              \end{array}
            \right.
\end{equation}
For $\alpha \notin B$  set $h(x, \alpha)=0$ for any $x \in G$.

 Then $h(x,\alpha)$ is a bounded measurable function defined on $G \times G$ and
\begin{equation}\label{D5a}
\int_G h(x, \alpha) dm(x)=0 \text{ for any } \alpha \in G.
\end{equation}
From $k>K$ and $x \in L_{k \cdot K}(f)$ it follows that $x \in H_{\alpha}+n_k\alpha$ for any $\alpha \in B$.
This implies
\begin{equation}\label{D6c}
h(x-n_k \alpha, \alpha)=1 \text{ for any } x\in L_{k \cdot K}(f) \text{ and } \alpha \in B.
\end{equation}
Taking average
\begin{equation}\label{D6a}
\frac{1}{m(B)} \int_B h(x-n_k\alpha,\alpha) dm(\alpha)=1 \text{ for } k>K \text{ and } x \in L_{k \cdot K}(f).
\end{equation}
Keep $\alpha$ fixed and select a character $\gamma \in \widehat{G}$. Consider in the Fourier-series of $h(x, \alpha)$ the coefficient $c_{\gamma}(\alpha)$ corresponding to this character, that is,

\begin{equation}\label{D6b}
c_{\gamma}(\alpha)=\int_G h(x,\alpha) \gamma(-x)dm(x).
\end{equation}
Since $h(x,\alpha)$ is a bounded measurable function, the function $c_{\gamma}(\alpha)$ is also bounded and measurable.
Then
\begin{equation}\label{D7a}
h(x,\alpha) \sim \sum_{\gamma \in \widehat{G}} c_{\gamma}(\alpha)\gamma(x).
\end{equation}
If $\gamma_0(x)\equiv 1$ then by \eqref{D5a} we have  
\begin{equation}\label{D7d}
c_{\gamma_0}(\alpha)=0 \text{ for any } \alpha \in G.
\end{equation}
$\text{ For a fixed }\alpha \in B \text{ we have }$
\begin{equation}\label{D7b}
h(x-n_k \alpha, \alpha) \sim \sum_{\gamma \in \widehat{G}} c_{\gamma} (\alpha) \gamma (-n_k \alpha) \gamma(x).
\end{equation}
By \eqref{D6a}
\begin{equation}\label{D7c}
m(L_{k \cdot K}(f)) \leq \int_G \left| \frac{1}{m(B)} \int_B h(x-n_k \alpha, \alpha)dm(\alpha)\right|^2dm(x)
\end{equation}
$$=\int_G|\varphi_k(x)|^2dm(x)=\circledast,$$
where $\ds \varphi_k(x)=\frac{1}{m(B)} \int_B h(x-n_k\alpha,\alpha)dm(\alpha)$ is a bounded measurable function. If $\gamma$ is a given character then using that $h$ is bounded and recalling \eqref{D6b} we obtain
\begin{equation}\label{F8a}
\begin{split}
&\widehat{\varphi}_k(\gamma)=\int_G \frac{1}{m(B)} \int_B h(x-n_k \alpha, \alpha) dm(\alpha) \gamma(-x)dm(x)
\\
&=\frac{1}{m(B)}\int_B  \int_G h(x-n_k \alpha, \alpha) \gamma(-x)dm(x)dm(\alpha)\\
&=\frac{1}{m(B)}\int_G \chi_B(\alpha)  \int_G h(u, \alpha) \gamma(-u-n_k \alpha)dm(u)dm(\alpha)\\
&=\frac{1}{m(B)}\int_G \chi_B(\alpha)  \gamma(-n_k \alpha)\int_G h(u, \alpha) \gamma(-u)dm(u)dm(\alpha)\\
&=\frac{1}{m(B)}\int_G \chi_B(\alpha)  \gamma(-n_k \alpha) c_{\gamma}(\alpha)dm(\alpha).
\end{split}
\end{equation}

By using the Parseval formula we can continue $\circledast$ in \eqref{D7c} to obtain
\begin{equation}\label{D9a}
\begin{split}
&m(L_{k \cdot K} (f)) \leq \sum_ {\gamma \in \widehat{G}}|\widehat{\varphi}_k(\gamma)|^2\\
&=\sum_{\gamma \in \widehat{G}} \frac{1}{(m(B))^2}\left|\int_G \chi_B(\alpha) \gamma (-n_k \alpha) c_{\gamma}(\alpha)dm(\alpha)\right|^2\\
&=\frac{1}{(m(B))^2}\sum_{\gamma \in \widehat{G}} \left|\int_G \chi_B(\alpha)c_{\gamma}(\alpha) \gamma^{n_k} (-\alpha) dm(\alpha)\right|^2.
\end{split}
\end{equation}
Since $\chi_B(\alpha)c_{\gamma}(\alpha)$ is a bounded measurable function and $\gamma^{n_k} \in \widehat{G}$, the expression $\int_G \chi_B(\alpha)c_{\gamma}(\alpha) \gamma^{n_k}(-\alpha)dm(\alpha)$ is a Fourier coefficient of this function. 

Now we use that $G$ is connected and hence $\widehat{G}$ is torsion-free. If $\gamma^{n_k}=\gamma^{n_{k'}}$ then $\gamma^{n_k-n_{k'}}=\gamma_0\equiv1$, but $\gamma$ is of infinite order and hence it is only possible if $n_k-n_{k'}=0$, that is $k=k'$. Hence for $k \neq k'$ the characters $\gamma^{n_k}$ and $\gamma^{n_{k'}}$ are different. By Parseval's formula for a fixed $\gamma \in \widehat{G}$
\begin{equation}\label{D10a}
\sum_{k=K}^{\oo}\left|\int_G \chi_B (\alpha)c_{\gamma}(\alpha) \gamma^{n_k}(-\alpha)dm(\alpha)\right|^2 \leq \int_G\left| \chi_B (\alpha)c_{\gamma}(\alpha)\right|^2dm(\alpha).
\end{equation}
This, Parseval's formula, \eqref{D5b}, \eqref{D6b} and \eqref{D9a} yield
\begin{equation}\label{D11a}
\begin{split}
&\sum_{k=K+1}^{\infty}m(L_{k \cdot K} (f))\leq \frac{1}{(m(B))^2} \sum_{\gamma \in \widehat{G}} \int_G \left|\chi_B(\alpha) c_{\gamma}(\alpha)\right|^2dm(\alpha)\\
&=\frac{1}{(m(B))^2}  \int_G \chi_B(\alpha)\sum_{\gamma \in \widehat{G}} \left|c_{\gamma}(\alpha)\right|^2dm(\alpha)\\&=\frac{1}{(m(B))^2}  \int_G \chi_B(\alpha)\int_G|h(x,\alpha)|^2dm(x)dm(\alpha)<\infty.
\end{split}
\end{equation}

Since $\ds \int_G|f| \leq K \cdot \sum_{k=0}^{\infty} m(L_{k \cdot K}(f))$ from \eqref{D11a} and $m(G)=1$ it follows that $f \in L^1(G)$. 

This completes the proof of the case of connected $G$. 

Next we show how  one can reduce the case of a locally connected $G$ to the connected case. If $G$ is locally connected then by \cite[24.45]{HR} if $C$ denotes the component of $G$ containing $O_G$ (the neutral element of $G$) then $C$ is an open subgroup of $G$ and $G$ is topologically isomorphic to $C \times (G/C)$. Since $G$ is compact $G/C$ should be finite. Suppose that its order is $n$. Using that $G=C \times (G/C)$ we write the elements of $G$ in the form $g=(g_1,g_2)$ with $g_1 \in C, \: g_2 \in G/C.$ 

Suppose that $f \notin L^1(G)$ is measurable and $m(\Gamma_{f,b})>0$. Set $$X_{\alpha,f}=\Big \{x \in G: \limsup_{k \rightarrow +\infty} \frac{|f(x+n_k \alpha)|}{k}< +\infty\Big \}.$$
If $\alpha \in \Gamma_{f,b}$ then $m(X_{\alpha,f})=1$. Suppose that $g^*_j, $ $ j=1,\dots,n$ is a list of all elements of $G/C$.

 For $x=(x_1,x_2)\in G$ define $$f^*(x)=f^*(x_1,x_2)=\sum_{j=1}^n|f(x_1, x_2+g^*_j)|.$$
Set $$X^*_{\alpha,f}= \bigcap_{j=1}^n\Big (X_{\alpha,f}+(0_C,g_j^*)).$$ Clearly $m(X_{\alpha,f})=1$ implies $m(X_{\alpha,f}^*)=1$.

 For $x \in X_{\alpha,f}^*$ we have $\ds \limsup_{k \rightarrow \infty} \frac{|f^*(x+n_k \alpha)|}{k} < +\infty.$ Since $f^*$ is not depending on its second coordinate we have $f^*(x+n_k(\alpha_1,\alpha_2))=f^*(x+n_k(\alpha_1,0_{G/C}))$. Define $f^{**}: C \rightarrow \R$ such that $f^{**}(x_1)=f^*(x_1, 0_{G/C})$. Since we assumed that $f \notin L^1(G)$ we have $f^* \notin L^1(G)$ and this implies $f^{**} \notin L^1(C)$.

 Set $$\Gamma_{f,b}^*=\pi_C(\Gamma_{f,b})=\{ \alpha_1: \exists \alpha_2 \in G/C \text{   such that }\aaa=(\alpha_1,\alpha_2) \in \Gamma_{f,b}\}.$$ Then for $\alpha_1 \in \Gamma_{f,b}^*$  we have
\begin{equation}\label{F12a}\limsup_{k \rightarrow \infty} \frac{|f^{**}(x_1+n_k \alpha_1)|}{k} < + \infty.
\end{equation}

Since the Haar measure on $C$ is a positive constant multiple of the Haar measure on $G$ 
restricted to $C$,
on the compact connected Abelian group $C$ we would obtain a measurable function $f^{**} \notin L^1(C)$ such that for a set of positive measure of rotations \eqref{F12a} holds. This would contradict the first part of this proof concerning connected groups.
\end{proof}
Theorem \ref{thconn} says that if we do not have ``too much  torsion" in $\widehat{G}$ then from $m(\Gamma_{f,b})>0$ it follows that $f \in L^1(G)$. In the next definition we define what we mean by ``a lot of torsion" in a group.

\begin{definition}\label{deftor}
We say that the group $G$ contains infinitely many multiple torsion if
\begin{enumerate}
  \item  either there is a prime number $p$ such that $G$ contains a subgroup algebraically isomorphic to the direct sum $(Z/p)\bigoplus(Z/p)\bigoplus \dots$ (countably many copies of $Z/p$),
  \item or there are infinitely many different prime numbers $p_1, p_2, \dots$ such that $G$ contains for any $j$ subgroups of the form $(Z/p_j) \times (Z/p_j)$.
\end{enumerate}
\end{definition}

\begin{theorem}\label{thtorsion}
Suppose that $(n_k)$ is a strictly monotone increasing sequence of integers and $G$ is a compact Abelian group such that its dual group $\widehat{G}$ contains infinitely many multiple torsion. Then there exists a measurable $f \notin L^1(G)$ such that
\begin{equation}\label{D16a}
m(\Gamma_{f,0})=m(\Gamma_{f,b})=1, \text{ where m is the Haar-measure on G. }
\end{equation}
In fact, we show that 
$\Gamma_{f,0}=\Gamma_{f,b}=G.$
\end{theorem}

\begin{proof}
First suppose that in Definition \ref{deftor} property (i) holds for $\widehat{G}$. Then for any $k$ we can select a subgroup $\widehat{G}_k$ in $\widehat{G}$ such that it is isomorphic to $\underbrace{(Z/p) \times (Z/p) \times \dots \times (Z/p)}_{\text{$k$ many times}}$. Suppose that the characters $\gamma_1, \dots, \gamma_k$ are the generators of $\widehat{G}_k$.

 Put $H_k= \bigcap_{j=1}^k \gamma_j^{-1}(1)$. Then $H_k$ is a closed subgroup of $G$. Since $y \in x+H_k$, that is $y-x \in H_k$ if and only if $\gamma_j(y)=\gamma_j(x)$ for $j=1, \dots, k$, which means that $\gamma_j(y-x)=\gamma_j(y)/\gamma_j(x)=1$ for $j=1, \dots, k$ one can see that $G$ is tiled with $p^k$ many translated copies of $H_k$. The sets $x+H_k$ are all closed and therefore $H_k$ is a closed-open subgroup of $G$. 
 
 We also have
\begin{equation}\label{D17a}
m(H_k)=\frac{1}{p^k}.
\end{equation}
Set $f_k(x)=p^k$ if $x \in H_k$ and $f_k(x)=0$ otherwise. 

Put $f=\sum_{k=1}^{\infty} f_k$. By the Borel-Cantelli lemma and \eqref{D17a} the function $f$ is $m$ a.e. finite. It is also clear that $f$ is measurable and $f \notin L^1(G)$. 

Suppose $\alpha \in G$  is arbitrary. Set $X_k=\bigcup_{j=0}^{p-1} H_k-j\alpha$. Then $m(X_k)=p^{-k+1}$ and by the Borel-Cantelli lemma $m$ a.e. $x$ belongs to only finitely many $X_k$. If $x \notin X_k$ then $\forall j \in \N, \: x+j\alpha \notin H_k$ and hence
\begin{equation}\label{D18a}
f_k(x+j\alpha)=0 \text{ for any } j \in \N.
\end{equation}
Therefore, $\frac{f(x+n_k \alpha)}{k} \rightarrow 0$ for $m$ a.e.  $x \in G$ and  $\Gamma_{f,0}=G.$ 

If in Definition \ref{deftor} property (ii) holds for $\widehat{G}$ then for any $k$ select $\widehat{G}_k$ in $\widehat{G}$ such that it is isomorphic to $(Z/p_k) \times (Z/p_k)$. We suppose that $\gamma_{1,k}$ and $\gamma_{2,k}$ are the generators of $\widehat{G}_k$. Put $H_k=\gamma_{1,k}^{-1}(1) \cap \gamma_{2,k}^{-1}(1).$ One can see, similary to the previous case, that $G$ is tiled by $p_k^2$ many translated copies of $H_k$. Turning to a subsequence if necessary, we can suppose that
\begin{equation}\label{D19a}
\sum_{k=1}^{\infty} \frac{1}{p_k} <+\infty.
\end{equation}
We also have
\begin{equation}\label{D19b}
m(H_k)=\frac{1}{p_k^2}.
\end{equation}
Set $f_k(x)=p_k^2$ if $x \in H_k$ and $f_k(x)=0$ otherwise. 

Put $f= \sum_{k=1}^{\infty} f_k$. Again, it is clear that $f$ is $m$ a.e. finite, measurable and $f \notin L^1(G)$. For an arbitrary $\alpha \in G$ one can define $X_k=\cup_{j=0}^{p_k-1} H_k-j \alpha$. Then $m(X_k)=\frac{1}{p_k}$.

 From \eqref{D19a} and from the Borel-Cantelli lemma it follows that $m$ a.e. $x$ belongs to only finitely many $X_k$. One can conclude the proof as we did it in the previous case.
\end{proof}
It is natural to ask for a version of Theorem \ref{thtorsion} for the non-conventional ergodic averages with $m(\Gamma_f)=1$ in \eqref{D16a}. For convergence of the non-conventional ergodic averages some arithmetic assumptions about $n_k$ are needed. \\
We recall from \cite{RW}  Definition 1.2 with some notational adjustment. 
\begin{definition}\label{ergoseq}. The sequence $(n_k)$ is ergodic ${\mathrm{mod}} \: q$ if for any $h \in \Z$
\begin{equation}\label{F18a}
\lim_{N \rightarrow \infty} \frac{\sum_{k=0}^N \chi_{h,q} (n_k)}{N+1} = \frac{1}{q},
\end{equation}
Where $\chi_{h,q}(x)=1$ if $x \equiv h \: {\mathrm{mod}} \: q$ and $\chi_{h,q}(x)=0$ otherwise.\\
A sequence $(n_k)$ is ergodic for periodic systems if it is ergodic ${\mathrm{mod}} \: q$ for every $q \in \N$.
\end{definition}

For ergodic sequences with essentially the same proof we can state the following version of Theorem \ref{thtorsion}:
\begin{theorem}\label{thtorsion2}
Suppose that $n_k$ is a strictly monotone, ergodic sequence for periodic systems and $G$ is a compact Abelian group such that its dual group $\widehat{G}$ contains infinitely many multiple torsion. Then there exists a measurable $f \notin L^1(G)$such that $\Gamma_f=G$, and hence $m(\Gamma_f)=1$.
\end{theorem}
\begin{proof}
As we mentioned earlier the argument of the proof of Theorem \ref{thtorsion} is applicable. One needs to add the observation that if $x\in X_k$ then the ergodicity of $n_k$ for periodic systems implies that $M_N^{\alpha} f_k$ converges. If $x \notin X_k$ then \eqref{D18a} can be used. Hence $M_N^{\alpha} f$ converges for all $\alpha \in G$ for a.e. $x$.
\end{proof}

In Theorem \ref{thtorsion} we saw that if there is ``lots of torsion" in $\widehat{G}$, that is, $G$ is "highly disconnected" then there are measurable functions $f$ not in $L^1$ for which $m(\Gamma_{f,0})=1$. Since the $p$-adic integers, $Z_p$ are the building blocks of $0$-dimensional compact Abelian groups (\cite[Theorem 25.22]{HR}) it is natural to consider them. If we take a countable product of $Z_p$ with $p$ fixed then the dual group will be the direct sum of $Z(p^{\infty})$'s  and will contain a subgroup algebraically isomorphic to the direct sum $(Z/p)\bigoplus(Z/p)\bigoplus \dots$. Then Theorem \ref{thtorsion} is applicable.

 If one considers an individual $Z_p$ then its dual group is $Z(p^{\infty})$ with all elements of finite order, so still there seems to be ``lots of torsion" in the dual group. It is also clear that arithmetic properties of $n_k$ might matter if we consider $Z_p$. For us it was quite surprising that if one considers ordinary ergodic averages, that is, $n_k=k$ then $Z_p$ behaves like a locally connected group and the following theorem is true.
\begin{theorem}\label{thpadic}
Suppose that $n_k=k$, and $p$ is a fixed prime number. We consider $G=Z_p$, the group of $p$-adic integers. Then for any measurable function $f:G \rightarrow \R$ from $m(\Gamma_{f,b})>0$ it follows that $ f \in L^1(G)$.
\end{theorem}
Before turning to the proof of Theorem \ref{thpadic} we need some notation and a Claim simplifying the proof of Theorem \ref{thpadic}. Denote by $\Gamma_{f,b}^j, \: j=-1,0,1, \dots$ the set of those $\alpha=(\alpha_0, \alpha_1, \dots)\in \Gamma_{f,b}$ for which $\alpha_{j+1} \neq 0$ but $\alpha_0=\dots=\alpha_j=0$. From $m(\Gamma_{f,b})>0$ it follows that there exists $j_0$ such that $m(\Gamma_{f,b}^{j_0})>0$. Given a finite string $(x_0, \dots,x_j)$ we denote by $[x_0, \dots, x_j]$ the corresponding cylinder set in $G$, that is, $$[x_0, \dots, x_j]=\{(x_0', x_1', \dots)\in G: (x_0',\dots,x_j')=(x_0, \dots,x_j)\}.$$

\begin{claim}\label{gfb0}
If from $m(\Gamma_{f,b}^{-1})>0$ it follows that $f \in L^1(G)$, then Theorem \ref{thpadic} is also true.
\end{claim}

\begin{proof}
As mentioned above if $m(\Gamma_{f,b})>0$ then we can choose $j_0$ such that $m(\Gamma_{f,b}^{j_0})>0$. Then for $\alpha \in \Gamma_{f,b}^{j_0}$ for any cylinder $[x_0, \dots, x_{j_0}]$ we have $[x_0, \dots, x_{j_0}]+\alpha=[x_0, \dots, x_{j_0}]$. If $\sigma$ is the one-sided shift on $Z_p$, that is, $\sigma(x_0,x_1,\dots)=(x_1,\dots)$ then for $\alpha \in \Gamma_{f,b}^{j_0}$ we have $\sigma^{j_0+1}(x+\alpha)=\sigma^{j_0+1}x+\sigma^{j_0+1}\alpha$. 

For an $x' \in G$ we define the function $f_{x_0, \dots, x_{j_0}}(x')=f(x_0, \dots, x_{j_0}, x')$, where $(x_0, \dots, x_{j_0}, x')$ is the concatenation of the finite string $(x_0, \dots, x_{j_0})$ and $x' \in G=Z_p$. Then $\Gamma^{-1}_{f_{x_0,\dots,x_{j_0}},b}\supset \sigma^{j_0+1}(\Gamma_{f,b}^{j_0})$ and we can apply the Claim for $f_{x_0, \dots,x_{j_0}}$ to obtain that $f_{x_0, \dots,x_{j_0}} \in L^1(G)$, that is, $f \in L^1([x_0, \dots,x_{j_0}])$. Since this holds for any cylinder set $[x_0, \dots,x_{j_0}]$ we obtain that $f \in L^1(G)$.
\end{proof}

\begin{proof}[Proof of Theorem \ref{thpadic}]
By Claim \ref{gfb0} we can assume that $m(\Gamma^{-1}_{f,b})>0$. We need to adjust the proof of Theorem \ref{thconn} for the case of $G=Z_p$. The key difficulty is the torsion in $\widehat{G}=Z(p^{\infty})$ which makes it impossible to use a direct argument which lead to \eqref{D10a}. Anyway, we start to argue as in the proof of Theorem \ref{thconn}, keeping in mind that now $n_k=k$. We introduce the sets $G_{\alpha,K}, \: B \subset \Gamma_{f,b}^{-1},\: L_k(f)$ as in \eqref{D4a}, \eqref{D4b}, and \eqref{D4c}, respectively. We fix $K$ and define the set $H_{\alpha}$ and the auxiliary function $h(x,\alpha)$ as in \eqref{D5b}. We have \eqref{D5a} again. 

Our aim is to establish that for a suitable $\kappa_0$
\begin{equation}\label{E4a}
\sum_{\kappa\geq \kappa_0} p^{\kappa} m(L_{p^{\kappa+2}\cdot K}(f))< \infty.
\end{equation}

Suppose that the function $\varphi$ equals $p^{\kappa+3} \cdot K$ on $L_{p^{\kappa+2}\cdot K}(f)\backslash L_{p^{\kappa+3}\cdot K}(f), \: \kappa=\kappa_0,\kappa_0+1, \dots$ and equals $K \cdot p^{\kappa_0+2}$ on $G \backslash L_{p^{\kappa_0+2} \cdot K} (f)$. Then $\varphi \geq |f|$ and by \eqref{E4a}
 \begin{equation}\label{E4b}
 \int_{G} \varphi dm \leq K \cdot p^{\kappa_0+2} m(G)+ \sum_{\kappa=\kappa_0}^{\infty} p^{\kappa+3} \cdot K m(L_{p^{\kappa+2} \cdot K}(f))< +\infty.
 \end{equation}
This implies that $f \in L^1(G)$. 

Hence we need to establish \eqref{E4a}. Choose and fix $\kappa_0 \in \N$ such that $p^{\kappa_0}>K$ and suppose that $\kappa \geq \kappa_0$. 

Then, keeping in mind that $L_{k\cdot K}(f)\supset L_{p^{\kkk+2}\cdot K}(f)$ for 
$k\leq p^{\kkk+2}$ we have instead of \eqref{D6c}
\begin{equation}\label{D32a}
h(x-k\alpha,\alpha)=1 \text{ for any } \alpha \in B, \: K<k<p^{\kappa+2} \text{ and } x\in L_{p^{\kappa+2}\cdot K}(f).
\end{equation}
Let
\begin{equation}\label{D32b}
h_{\kappa}(x,\alpha)=\frac{1}{p^{\kappa}} \sum_{k=p^{\kappa}}^{2p^{\kappa}-1} h(x-k\alpha,\alpha).
\end{equation}
Then by \eqref{D32a}
\begin{equation}\label{D32c}
h_{\kappa}(x-k \alpha,\alpha)=1 \text{ for any } \alpha \in B, \: 0\leq k < p^{\kappa+2}-2p^{\kappa} \text{ and } x \in L_{p^{\kappa+2} \cdot K}(f) 
\end{equation}

Taking average on $B$
\begin{equation}\label{D33a}
\frac{1}{m(B)}\int_B h_{\kappa}(x-k\alpha,\alpha)dm(\alpha)=1 
\end{equation}
$$\text{ for } \kappa \geq \kappa_0, \: 0 \leq k < p^{\kappa+2}-2p^{\kappa} \text{ and } x\in L_{p^{\kappa+2} \cdot K} (f).$$
Now we return to $h(x,\alpha)$ and we define $c_{\gamma}(\alpha)$ as in \eqref{D6b}. Again, $c_{\gamma}(\alpha)$ is a bounded, measurable function and \eqref{D7a} holds.\\
Denoting again by $\gamma_0(x)$ the identically $1$ character, the neutral element of $\widehat{G}$ we also have \eqref{D7d} satisfied. For $h_{\kappa}(x,\alpha)$ we have
\begin{equation}\label{D33b}
h_{\kappa}(x,\alpha) \sim 
\sum_{\gamma \in \widehat{G}} c_{\gamma,\kappa}(\alpha)\gamma(x)=\sum_{\gamma \in \widehat{G}} c_{\gamma}(\alpha)\left(\frac{1}{p^{\kappa}}\sum_{k=p^{\kappa}}^{2p^{\kappa}-1} \gamma(-k \alpha)\right) \gamma(x).
\end{equation} 

Since $\widehat{G}=Z(p^{\infty})$, the order of $\gamma$ is a power of $p$. We denote it by ${\mathrm {ord}}(\gamma)$. A $\gamma \in \widehat{G}$ of order $p^{r}, \: r>0$ is of the form
\begin{equation}\label{D34b}
\begin{split}
&\gamma(x)=\exp\left(\frac{2 \pi i l}{p^{r}}(x_0+px_1+\dots+p^{r -1}x_{r-1})\right) \\ 
&\text{ for } x=(x_0,x_1,\dots)\in G=Z_p \text{ with $l$ not divisible by } p.
\end{split}
\end{equation}
Since $B \subset \Gamma_{f,b}^{-1}$, for $\alpha \in B$ we have $\alpha_0 \neq 0$ which implies $\gamma(-\alpha) \neq 1$ and if $\gamma$ is of order $p^{r}, \: r>0$ then $\ggg(-\aaa)\in\C$ is also of order $p^{r}$, $r>0$. Hence for ${\mathrm{ord}}(\gamma)=p^{r}\leq p^\kappa$ and $\alpha \in B$ we have
\begin{equation}\label{D37a}
\sum_{k=p^{\kappa}}^{2p^{\kappa}-1} \gamma(-k \alpha)=\sum_{k=p^{\kappa}}^{2p^{\kappa}-1} \gamma^k(-\alpha)=\gamma(-p^{\kappa} \alpha) \frac{1-\gamma^{ p^{\kappa}}(-\alpha)}{1-\gamma(-\alpha)}=0.
\end{equation}
This way we can get rid of some characters with small torsion in the Fourier-series of $h_{\kappa}(x,\alpha)$.

Recalling that $c_{\gamma_0}(\alpha)=\int_G h(x,\alpha) \cdot 1 dm(\alpha)=0$ by \eqref{D7a} we have in \eqref{D33b}
\begin{equation}\label{E7a}
c_{\gamma_{0},\kappa}(\alpha)=0 \text{ if }
\alpha \in B.
\end{equation} 

Using \eqref{D33b} again we have
\begin{equation}\label{D40b}
h_{\kappa}(x-k \alpha,\alpha) \sim \sum_{\gamma \in \widehat{G}} c_{\gamma,\kappa}(\alpha)\gamma(-k \alpha)\gamma(x)
\end{equation}
and by \eqref{D33a} for any $0\leq k<p^{\kappa+2}-2p^{\kappa}$
\begin{equation}\label{D40c}
m(L_{p^{\kappa+2} \cdot K}(f))\leq \int_G\left|\frac{1}{m(B)}\int_B h_{\kappa}(x-k \alpha,\alpha)dm(\alpha)\right|^2dm(x)
\end{equation}
$$=\int_G|\varphi_{\kappa,k}(x)|^2dm(x),$$
where $\ds \varphi_{\kappa,k}(x)=\frac{1}{m(B)}\int_B h_{\kappa}(x-k \alpha,\alpha) dm(\alpha)$ is a bounded measurable function. 

Recall that by \eqref{D33b} we can express the Fourier-coefficients of $h_{\kappa}$ by those of $h$, that is
\begin{equation}\label{D39a}
c_{\gamma,\kappa}(\alpha)=\int_G h_{\kappa}(x,\alpha)\gamma(-x)dm(x)=c_{\gamma}(\alpha)\frac{1}{p^{\kappa}}\sum_{k=p^{\kappa}}^{2p^{\kappa}-1} \gamma(-k \alpha).
\end{equation}
Therefore,
\begin{equation}\label{D41a}
\begin{split}
&\widehat{\varphi}_{\kappa,k}(\gamma)=\int_G\frac{1}{m(B)} \int_B h_{\kappa}(x-k\alpha,\alpha)dm(\alpha)\gamma(-x)dm(x)\\
&=\frac{1}{m(B)}\int_B\int_G h_{\kappa}(x-k\alpha,\alpha)\gamma(-x)dm(x)dm(\alpha)\\
&=\frac{1}{m(B)}\int_G \chi_B(\alpha) \cdot \int_G h_{\kappa}(u,\alpha)\gamma(-u-k\alpha)dm(u)dm(\alpha)\\
&=\frac{1}{m(B)}\int_G \chi_B(\alpha)\gamma(-k \alpha) c_{\gamma,\kappa}(\alpha)dm(\aaa).
\end{split}
\end{equation}

If $\ggg\not=\ggg_{0}$ and $\mathrm{ord}(\ggg)\leq p^{\kkk}$
then by \eqref{D37a} and \eqref{D39a} we have $c_{\ggg,\kkk}(\aaa)=0$
for any $\aaa\in B$, and hence $\widehat{\varphi}_{\kappa,k}(\gamma)=0$.

Recall from \eqref{E7a} that if $\alpha \in B$ 
then $c_{\gamma_{0},\kappa}(\alpha)=0$. Hence $\widehat{\varphi}_{\kappa,k}(\gamma_{0})=0$ holds in this case as well.

Now suppose that $\gamma^{p^{\kappa}} \neq \gamma_0$. Then ${\mathrm{ord}}(\gamma) \geq p^{\kappa+1}$ and for $k=0,\dots,p^{\kappa+1}-1$ the characters $\gamma^{k}$ are different. 

By using the Parseval-formula we can continue \eqref{D40c} to obtain
for any $0\leq k <p^{\kkk+2}-2p^{\kkk}$ that
\begin{equation}\label{D42a}
m(L_{p^{\kappa+2}\cdot K}(f))\leq \sum_{\gamma \in \widehat{G}}|\widehat{\varphi}_{\kappa,k}(\gamma)|^2
\end{equation}
$$=\sum_{\gamma \in \widehat{G}, \: \gamma^{p^{\kappa}}\neq \gamma_0}\frac{1}{(m(B))^2} \cdot \left|\int_G \chi_B(\alpha)\gamma(-k \alpha)c_{\gamma,\kappa}(\alpha)dm(\alpha)\right|^2.$$

Since $p\geq 2$ implies $p^{\kkk+2}\geq 3p^{\kkk}$ 
we can use  \eqref{D32c} and \eqref{D42a} 
for $k=0,...,p^{\kkk}-1.$
Adding equation \eqref{D42a} for all $\kkk\geq \kkk_{0}$ and $k=0,...,p^{\kkk}-1$ we need to estimate 
\begin{equation}\label{D45a} 
\sum_{\kappa \geq \kappa_0} p^{\kappa}m(L_{p^{\kappa+2}\cdot K}(f))
\end{equation}
$$\leq \sum_{\kappa \geq \kappa_0}\sum_{\gamma \in \widehat{G}, \gamma^{p^{\kappa}}\neq \gamma_0} \frac{1}{(m(B))^2} \sum_{k=0}^{p^{\kappa}-1}\left|\int_G \chi_B(\alpha)c_{\gamma,\kappa}(\alpha)\gamma(-k \alpha) dm(\alpha)\right|^2.$$

Using \eqref{D33b} and  \eqref{D39a} first we estimate for $\kappa \geq \kappa_0$
\begin{equation}\label{E10a}
\sum_{k=0}^{p^{\kappa}-1} \left|\int_G \chi_B (\alpha) c_{\gamma,\kappa}(\alpha)\gamma(-k \alpha)dm(\alpha)\right|^2
\end{equation}
$$=\sum_{k=0}^{p^{\kappa}-1}\left|\int_G \chi_B(\alpha)c_{\gamma}(\alpha) \frac{1}{p^{\kappa}} \sum_{k'=p^{\kappa}}^{2p^{\kappa}-1}\gamma(-(k'+k)\alpha)dm(\alpha)\right|^2=**$$
in the last expression $k'+k$ can take values between $p^{\kappa}$ and $3p^{\kappa}-2$. If $p\geq 3$ then $3p^{\kappa}-2 \leq p^{\kappa+1}-1$ so for the moment we suppose that $p \geq 3$. In the end of this proof we will point out the little adjustments which we need for the case $p=2$.

For $p^{\kappa}\leq j \leq 3p^{\kappa}-2\leq p^{\kappa+1}-1$ we denote by $w_j'$ the number of those couples $(k,k')$ for which $ 0 \leq k \leq p^{\kappa} -1, \: p^{\kappa} \leq k'\leq 2p^{\kappa}-1$ and $k+k'=j$.\\
Obviously, $w_j'\leq p^{\kappa}$. Set $w_j=w_j'/p^{\kappa}\leq 1.$ We select these $w_j$ for all $\kappa_0 \leq \kappa \leq {\mathrm{ord}}(\gamma)$. For those values of $j$ for which we have not defined $w_j$ yet we set $w_j=0$. 

By using this notation we can continue $**$ from \eqref{E10a}
\begin{equation}\label{D46a}
** \leq \sum_{j=p^{\kappa}}^{p^{\kappa+1}-1} w_j \left|\int_G \chi_B(\alpha) c_{\gamma}(\alpha) \cdot \gamma(-j \alpha)dm(\alpha)\right|^2 
\end{equation}
$$\leq \sum_{j=p^{\kappa}}^{p^{\kappa+1}-1} \left|\int_G \chi_B(\alpha)c_{\gamma}(\alpha) \cdot \gamma(-j \alpha) dm(\alpha)\right|^2.$$
Using \eqref{E10a} and \eqref{D46a} while continuing the estimation of \eqref{D45a} we obtain
\begin{equation}\label{D47a}
\begin{split}
&\sum_{\kappa \geq \kappa_0} p^{\kappa}m(L_{p^{\kappa+2}\cdot K}(f))\leq\\
& \leq \sum_{\kappa \geq \kappa_0} \sum_{\gamma \in \widehat{G}, \gamma^{p^{\kappa}}\neq \gamma_0} \frac{1}{(m(B))^2}\sum_{j=p^{\kappa}}^{p^{\kappa+1}-1}\left|\int_G \chi_B(\alpha)c_{\gamma}(\alpha)\gamma(-j \alpha) dm(\alpha)\right|^2 \\ 
&\leq \sum_{\gamma \in \widehat{G}} \sum_{j=1}^{{\mathrm{ord}}(\gamma)-1} \frac{1}{(m(B))^2} \cdot \left| \int_G \chi_B (\alpha) c_\gamma(\alpha) \gamma(-j\alpha)dm(\alpha)\right|^2.
\end{split}
\end{equation}
Since for a fixed $\gamma$ the characters $\gamma^{-j}$ are different, for different values $0 \leq j < {\mathrm{ord}}(\gamma)$ by Parseval's Theorem we infer
\begin{equation}\label{D47b}
\sum_{j=1}^{{\mathrm{ord}}(\gamma)-1} \left| \int_G \chi_B(\alpha) c_{\gamma}(\alpha) \gamma(-j\alpha)dm(\alpha)\right|^2 \leq \int_G\left|\chi_B(\alpha)c_{\gamma}(\alpha)\right|^2dm(\alpha).
\end{equation}
Using this in \eqref{D47a} we obtain
\begin{equation}\label{D48a}
\begin{split}
&\sum_{\kappa \geq \kappa_0} p^{\kappa}m(L_{p^{\kappa+2}\cdot K}(f)) \leq \frac{1}{(m(B))^2} \sum_{\gamma \in \widehat{G}} \int_G |\chi_B(\alpha)c_{\gamma}(\alpha)|^2dm(\alpha)\\
&=\frac{1}{(m(B))^2}\int_G\chi_B(\alpha) \sum_{\gamma \in \widehat{G}} |c_{\gamma}(\alpha)|^2dm(\alpha)\\&=\frac{1}{(m(B))^2}\int_G\chi_B(\alpha)\int_G |h(x,\alpha)|^2dm(x)dm(\alpha)<+\infty.
\end{split}
\end{equation}
This completes the proof if $p \geq 3$. 

In case of $p=2$ the intervals $p^{\kappa} \leq j \leq 3p^{\kappa}-2$ are not disjoint, but $3p^{\kappa}-2 \leq p^{\kappa+2}-1$.
Instead of \eqref{D47a} we could obtain
$$\sum_{\kappa \geq \kappa_0} p^{\kappa+1}m(L_{p^{\kappa+1} \cdot K} (f))\leq 2 \cdot \sum_{\gamma \in \widehat{G}} \sum_{j=1}^{2 {\mathrm{ord}}(\gamma)-1} \frac{1}{(m(B))^2}\left|\int_G \chi_B(\alpha) c_{\gamma}(\alpha) \gamma(-j \alpha)dm(\alpha)\right|^2.$$

For a fixed $\gamma$ the characters $\gamma^{-j}(\alpha), \: j \leq 2 {\mathrm{ord}}(\gamma)-1$ are not different but for each $j \leq 2 {\mathrm{ord}}(\gamma)-1$ there is at most one other $j'\leq2 {\mathrm{ord}}(\gamma)-1$ such that $\gamma^{-j}=\gamma^{-j'}$, hence $$\sum_{j=1}^{2 {\mathrm{ord}}(\gamma)-1}\left| \int_G \chi_B (\alpha) c_\gamma(\alpha)\gamma(-j \alpha)dm(\alpha)\right|^2 \leq 2 \int_G |\chi_B(\alpha)c_{\gamma}(\alpha)|^2dm(\alpha).$$
The conclusion of the proof is similar to the $p\geq3$ case.
\end{proof}

\end{document}